\documentclass[a4]{amsart}

\input xypic 
\input xy 
\xyoption{all} 
\usepackage{amssymb} 
\usepackage{bbm}
\usepackage{tikz-cd}

\usepackage{float}
\usepackage{graphicx}
\usepackage{accents}
\usepackage[bookmarksnumbered, bookmarksopen, 
colorlinks,citecolor=blue,linkcolor=blue,backref]{hyperref}

\newtheorem{theorem}{Theorem}[section]
\newtheorem{proposition}[theorem]{Proposition}
\newtheorem{lemma}[theorem]{Lemma}
\newtheorem{corollary}[theorem]{Corollary}

\theoremstyle{definition}
\newtheorem{definition}[theorem]{Definition}
\newtheorem{remark}[theorem]{Remark}

\newtheorem{conjecture}[theorem]{Conjecture}

\newcounter{RomanNumber}

\newcommand{\calP}{\ensuremath{\mathcal{P}}} 
\newcommand{\calQ}{\ensuremath{\mathcal{Q}}} 

\newcounter{bean}
\newenvironment{letterlist}{\begin{list}{\rm ({\alph{bean}})}
      {\usecounter{bean}\setlength{\rightmargin}{\leftmargin}}}
      {\end{list}}

\newcommand{\nameddright}[5]{\ensuremath{#1\stackrel{#2}
 {\longrightarrow}#3\stackrel{#4}{\longrightarrow}#5}}

\newcommand{\larrow}{\relbar\!\!\relbar\!\!\rightarrow}
\newcommand{\llarrow}{\relbar\!\!\relbar\!\!\larrow}
\newcommand{\lllarrow}{\relbar\!\!\relbar\!\!\llarrow}

\newcommand{\llnameddright}[5]{\ensuremath{#1\stackrel{#2}
 {\llarrow}#3\stackrel{#4}{\llarrow}#5}}

\newcommand{\qqed}{\hfill\Box}

\begin{document}

%%% Title

\title{Local hyperbolicity, inert maps and applications in Moore's conjecture} 

\author{Ruizhi Huang} 
\address{State Key Laboratory of Mathematical Sciences \& Institute of Mathematics, Academy of Mathematics and Systems Science, 
   Chinese Academy of Sciences, Beijing 100190, China} 
\email{huangrz@amss.ac.cn} 
   \urladdr{https://sites.google.com/site/hrzsea/}

\subjclass[2010]{Primary 
55Q52, %Homotopy groups of special spaces
55P62; %rational homotopy theory
Secondary 
55P35, %Loop spaces
55P40, %Suspensions
%57N65; %Algebraic topology of manifolds 
}
\keywords{hyperbolic, Moore's conjecture, inert maps}

%%% Abstract

\begin{abstract} 
We show that the base space of a homotopy cofibration is locally hyperbolic under various conditions. In particular, if these manifolds admit a rationally elliptic closure, then almost all punctured manifolds and almost all manifolds with rationally spherical boundary are $\mathbb{Z}/p^r$-hyperbolic for almost all primes $p$ and all integers $r \geq 1$, and satisfy Moore's conjecture at sufficiently large primes.
\end{abstract}

\maketitle

%\setcounter{tocdepth}{1}
%\tableofcontents

\section{Introduction} 

Let $X$ be a simply connected space of finite type with finite rational Lusternik-Schnirelmann category. The classical rational dichotomy~\cite{FHT82, FHT01} asserts that $X$ is either
\begin{itemize}
  \item[-] \emph{rationally elliptic}, meaning that 
    $\pi_\ast(X) \otimes \mathbb{Q}$ is finite dimensional, or
  \item[-] \emph{rationally hyperbolic}, meaning that 
    $\pi_\ast(X) \otimes \mathbb{Q}$ grows exponentially. 
\end{itemize}

Let $p$ be a prime. A space $X$ has \emph{homotopy exponent} $p^{r}$ 
if $r$ is the least power of $p$ that annihilates the $p$-torsion in $\pi_{\ast}(X)$. In this case, 
write $\exp_{p}(X)=p^{r}$. If there is no such power of $r$, that is, if $\pi_{\ast}(X)$ has  
$\mathbb{Z}/p^{r}$ summands for arbitrarily large $r$, write $\exp_{p}(X)=\infty$. 

Moore's conjecture posits a profound relationship between homotopy exponents and the rational dichotomy. \begin{conjecture}[Moore]  
\label{Moorehyper}
Let $X$ be a simply-connected finite $CW$-complex. Then the following are equivalent: 
\begin{letterlist} 
   \item $X$ is rationally hyperbolic; 
   \item $\exp_{p}(X)=\infty$ for some prime $p$; 
   \item $\exp_{p}(X)=\infty$ for all primes $p$. 
\end{letterlist} 
\end{conjecture} 
The conjecture can be equivalently reformulated in terms of rationally elliptic spaces. Additionally, to formulate a weak version of Moore's conjecture, statement (c) may be relaxed to require that $\exp_{p}(X)=\infty$ for all but finitely many primes $p$.

Moore's conjecture has been proved in certain special cases but is very far from being resolved. For instance, it has been verified for spheres and for mod-$p^r$ Moore spaces with $p>2$ or $r>1$ \cite{CMN79a, CMN79b, Nei81, Sel84, Nei87, Coh89}. However, the conjecture remains open for mod-$2$ Moore spaces, which represent the first genuinely difficult case, even though the growth of their exponents has been intensively studied \cite{CW95, ChW13, MW16}. For finite complexes at large primes, Moore's conjecture often reduces to the cases of Moore spaces and spheres via loop space decompositions \cite{MW86, Ani89}. Relevant research on Moore's conjecture in special cases can be found in classical references \cite{Lon78, NS82, Sel83, Sta02, Ste04, CPSS08}. More recently, intensive progress on the Moore's conjecture has been made for Poincar\'{e} duality complexes \cite{BW15, BB18, BT22, HT22, Hua23, HT23, The24a, The24b, Hua25} and generalized moment-angle complexes \cite{HST19, Kim18}, by combining loop space decomposition techniques with cohomology ring or combinatorial structures. In addition, unusual connections between homotopy groups and group theory \cite{Coh95, Wu01, CW04, Wu10, CW11, LW11, BW13, CMW18, HW20b, BM23} might give new clues for attacking Moore's conjecture in an unexpected way.

 Moore's conjecture asserts that rationally hyperbolic spaces have torsion homotopy groups 
of arbitrarily high order. But it says nothing about the rate of growth of the $p$-torsion 
in the homotopy groups. To address this, Huang and Wu~\cite{HW20a} introduced 
the notion of local hyperbolicity in analogy with rational hyperbolicity, and showed that Moore spaces are locally hyperbolic. This direction was subsequently explored in greater depth by Zhu-Pan \cite{ZP21}, Boyde \cite{Boy22, Boy24}, Huang-Theriault \cite{HT24} and Boyde-Huang \cite{BH24}.

\begin{definition}\label{hyperdef}
A $CW$-complex $X$ is called {\it $\mathbb{Z}/p^r$-hyperbolic} if the number of $\mathbb{Z}/p^r$-summands in $\pi_\ast(X)$ has exponential growth, that is,
\[
\liminf_n\frac{{\rm ln}~ t_n}{n}>0,
\]
where $t_n=\sharp ~\{ \text{$\mathbb{Z}/p^r$-summands in $\mathop{\bigoplus}\limits_{m\leq n} \pi_m(X)$}\}$.
\end{definition}

In this paper, we investigate local hyperbolicity and Moore's conjecture within the framework of homotopy cofibrations under an inertness assumption, and establish that the base space of a homotopy cofibration is locally hyperbolic. 
Let $A\stackrel{h}{\larrow} X\stackrel{\varphi}{\larrow} Y$ be a homotopy cofibration. Following \cite{FHT82, The24a, Hua24}, the map $h$ is said to be {\it inert} if $\Omega \varphi$ has a right homotopy inverse. The following result highlights a situation in which rational homotopy and local homotopy are closely intertwined.

\begin{theorem}\label{elliptic-hyper-thm-intro}
Let $\Sigma A\stackrel{h}{\larrow} X\stackrel{\varphi}{\larrow} Y$ be a homotopy cofibration of simply connected finite $CW$-complexes such that both $\Sigma A$ and $Y$ are not rationally contractible. Suppose that $h$ is rationally inert. If $Y$ is rationally elliptic, then $X$ is rationally hyperbolic and $\mathbb{Z}/p^r$-hyperbolic for almost all primes $p$ and all $r\geq 1$.

In particular, Moore's conjecture holds for $X$ for all but finitely many primes $p$.
\end{theorem}

In \cite[Corollary 5.3]{HT24}, the same conclusion is established under the assumption that $h$ is integrally inert. In contrast, Theorem \ref{elliptic-hyper-thm-intro} requires only that $h$ be rationally inert. It is important to note that rational inertness is a considerably weaker condition than integral or local inertness. The key contribution of Theorem \ref{elliptic-hyper-thm-intro} lies in deriving a local result from a purely rational assumption.

The concept of rational inertness is classical in rational homotopy theory \cite{FHT82} and was widely studied \cite{Ani86, FT89, HL95, HeL96, Bub05, FHT07}. A fundamental result of Halperin and Lemaire \cite{HL87} shows that the attaching map for the top cell of any Poincar\'{e} duality complex is rationally inert unless its rational cohomology algebra is generated by a single element. With this result, we can show the following theorem from Theorem \ref{elliptic-hyper-thm-intro}. 
For a simply-connected $m$-dimensional Poincar\'{e} Duality complex $M$, denoted by $\overline{M}$ its $(m-1)$-skeleton. 

\begin{theorem}\label{Melliptic-hyper-thm-intro}
Let $M$ be a simply connected Poincar\'{e} Duality complex such that its rational homology algebra is not generated by a single element. If $M$ is rationally elliptic, then $\overline{M}$ is rationally hyperbolic and $\mathbb{Z}/p^r$-hyperbolic for almost all primes $p$ and all $r\geq 1$.

In particular, Moore's conjecture holds for $\overline{M}$ for all but finitely many primes $p$.
\end{theorem}

Theorems \ref{elliptic-hyper-thm-intro} and \ref{Melliptic-hyper-thm-intro} provide supporting evidence for \cite[Conjecture 1.6]{HT24}, which claims that rational hyperbolicity should imply $\mathbb{Z}/p^r$-hyperbolicity for all primes $p$ and all $r\geq 1$. 

As an application of Theorems \ref{elliptic-hyper-thm-intro} and \ref{Melliptic-hyper-thm-intro}, two results concerning manifolds are established. A {\it punctured manifold} $M\backslash B$ is a manifold $M$ with a finite set $B$ of points removed.  

\begin{theorem}\label{punct-thm}
Let $M$ be a simply connected closed manifold such that its rational homology algebra is not generated by a single element. If $M$ is rationally elliptic, then its punctured manifold $M\backslash B$ is rationally hyperbolic and satisfies the following:
\begin{itemize}
\item when $\#B=1$ or $2$, $M\backslash B$ is $\mathbb{Z}/p^r$-hyperbolic for almost all primes $p$ and all $r\geq 1$;
\item when $ \#B\geq 3$, $M\backslash B$ is $\mathbb{Z}/p^r$-hyperbolic for all primes $p$ and all $r\geq 1$.
\end{itemize}

In particular, Moore's conjecture holds for $M\backslash B$ for all but finitely many primes $p$.
\end{theorem}

\begin{theorem}\label{NSthm}
Let $N$ be a simply connected manifold whose boundary is a simply connected rational sphere and whose rational homology algebra is not generated by a single element. If the homotopy cofibre of the inclusion of the boundary is rationally elliptic, then $N$ is rationally hyperbolic and $\mathbb{Z}/p^r$-hyperbolic for almost all primes $p$ and all $r\geq 1$.

In particular, Moore's conjecture holds for $N$ for all but finitely many primes $p$.
\end{theorem}

Theorem \ref{Melliptic-hyper-thm-intro} excludes the case when $M$ is rationally hyperbolic, for which we propose the following conjecture. 

\begin{conjecture}\label{conj-intro}
Let $M$ be a simply connected Poincar\'{e} Duality complex. If $M$ is rationally hyperbolic, then $\overline{M}$ is $\mathbb{Z}/p^r$-hyperbolic for almost all primes $p$ and all $r\geq 1$.
\end{conjecture}
If $M$ is rationally hyperbolic, the attaching map for the top cell of $M$ is rationally inert by \cite{HL87}, and then $\overline{M}$ is also rationally hyperbolic. Hence, Conjecture \ref{conj-intro} is a special case of \cite[Conjecture 1.6]{HT24}.

The preceding results concern homotopy cofibrations of finite $CW$-complexes under the rational inertness condition. We also establish results for homotopy cofibrations with infinite homotopy cofibres under the local inertness condition. It is worth noting that local inertness is a relatively recent concept, introduced and studied in the works of Theriault \cite{The24a, The24c} and the present author \cite{Hua24}.

For a path connected finite $CW$-complex $A$ of dimension $d$ and connectivity $s$, define two finite sets $\mathcal{P}(A)$ and $\mathcal{Q}(A)$ of primes by 
\[
\begin{split}
\calQ(A)&=\{ \text{prime}~q~|~s+2q-2<d\},\\
\calP(A)&=\calQ(x)\cup \{ \text{prime}~q~|~\text{$H_{\ast}(A;\mathbb{Z})$ has $q$-torsion} \}.
\end{split}
\]
Note that $\calP(A)$ is empty if $A$ is a sphere, and $\calQ(A)$ is empty if $A$ is a Moore space.

\begin{theorem}\label{gen-hyper-thm}
Let $\Sigma A\stackrel{h}{\larrow} X\stackrel{\varphi}{\larrow} Y$ be a homotopy cofibration of simply connected $CW$-complexes such that $\Sigma A$ is finite and not rationally contractible.   
Suppose that $h$ is inert after localization away from a finite set $P_1$ of primes. The following hold:
\begin{itemize}
\item[(1).] 
if ${\rm dim}(\widetilde{H}^\ast(A;\mathbb{Q}))\geq 2$, then $X$ is $\mathbb{Z}/p^r$-hyperbolic for all primes $p\notin \calP(A)\cup P_1$ and all $r\geq 1$;
\item[(2).]
if there is a map $S^t\stackrel{}{\larrow} Y$ whose loop has a left homotopy inverse after localization away from a finite set $P_2$ of primes, then $X$ is $\mathbb{Z}/p^r$-hyperbolic for all primes $p\notin \calP(A)\cup P_1\cup P_2$ and all $r\geq 1$.
\end{itemize}
If further $X$ is of finite type and has finite rational Lusternik-Schnirelmann category, then $X$ is rationally hyperbolic. In particular, Moore's conjecture holds for $X$ for all but finitely many primes $p$. 
\end{theorem}

Let $P^{k+1}(p^r)$ be the Moore space such that the reduced cohomology $\widetilde{H}^\ast(P^{k+1}(p^r);\mathbb{Z})$ is isomorphic to $\mathbb{Z}/p^r$ if $\ast=k+1$ and $0$ otherwise.

\begin{theorem}\label{gen-hyper-thm2}
Let $\Sigma A\stackrel{h}{\larrow} X\stackrel{\varphi}{\larrow} Y$ be a homotopy cofibration of simply connected $CW$-complexes of finite type such that $\Sigma A$ is finite and not rationally contractible. Suppose that $h$ is inert after localization at an odd prime $p$. If either of the following holds with $r\geq 1$:
\begin{itemize}
\item[(a).]
$p\notin \mathcal{P}(A)$, and there is a map $\mu: P^{k+1}(p^r)\stackrel{}{\larrow} Y$ inducing an injection on homology
\[
\mu_\ast: \widetilde{H}_\ast( P^{k+1}(p^r);\mathbb{Z}/p^r) \stackrel{}{\larrow}\widetilde{H}_\ast(Y; \mathbb{Z}/p^r); 
\]
\item[(b).] 
$p\notin \mathcal{Q}(A)$, there is a map $S^t\stackrel{}{\larrow} Y$ whose loop has a left homotopy inverse after localization at $p$, and there is an order $p^r$ element in $\widetilde{H}_\ast(A;\mathbb{Z})$,
\end{itemize}
then $X$ is $\mathbb{Z}/p^r$-hyperbolic.
\end{theorem}

The paper is organized as follows. In Section \ref{sec: lift} we prove two suspension splittings of finite $CW$-complexes into wedges of Moore spaces at large primes. In Section \ref{sec: gen} we prove Theorems \ref{gen-hyper-thm} and \ref{gen-hyper-thm2}. In Section \ref{sec: elliptic} we prove Theorems \ref{elliptic-hyper-thm-intro}, and in Section \ref{sec: ellipticII} we prove Theorems \ref{Melliptic-hyper-thm-intro}, \ref{punct-thm} and \ref{NSthm}.

$\, $

\noindent{\bf Acknowledgements.} 
The author was supported in part by the National Natural Science Foundation of China (Grant nos. 12331003 and 12288201), the National Key R\&D Program of China (No. 2021YFA1002300) and the Youth Innovation Promotion Association of Chinese Academy Sciences.

%---------------------------------------------------------------------------------------
\section{Lifting lemmas and two suspension splittings}
\label{sec: lift}
In this section, we prove some lifting lemmas by standard obstruction theory, and apply them to show two suspension splittings of finite $CW$-complexes into wedges of Moore spaces at large primes: Propositions \ref{Xwedge-prop1} and \ref{Xwedge-prop2}.

%----------------
\subsection{Connectivity estimations}
In this subsection, we compute the connectivity of certain maps that will be used in the sequel.  
Let $p$ be a prime, $r$, $n\geq 1$. Define the space $S^{2n+1}\{p^r\}$ by the homotopy fibration
\[
S^{2n+1}\{p^r\} \stackrel{}{\larrow} S^{2n+1} \stackrel{p^r}{\larrow} S^{2n+1},
\]
where $p^r$ is the degree $p^r$ self-map of $S^{2n+1}$. The space $S^{2n+1}\{p^r\}$ is $(2n-1)$-connected and $\pi_{2n}(S^{2n+1}\{p^r\}) \cong \mathbb{Z}/p^r$. Then there is a map 
\begin{equation}\label{ip-eq}
i_p: S^{2n+1}\{p^r\}  \stackrel{}{\larrow} K(\mathbb{Z}/p^r, 2n)
\end{equation}
representing a generator of $H^{2n}(S^{2n+1}\{p^r\} ;\mathbb{Z}/p^r)\cong \mathbb{Z}/p^r$.
\begin{lemma}\label{SKlemma}
The inclusion $S^{2n+1}\stackrel{i}{\larrow} K(\mathbb{Z}, 2n+1)$ of the bottom cell is $(2n+2p-2)$-connected after localization at $p$. 

The map $S^{2n+1}\{p^r\}  \stackrel{i_p}{\larrow} K(\mathbb{Z}/p^r, 2n)$ is $(2n+2p-3)$-connected.
\end{lemma}
\begin{proof}
For the first statement, it is well-known that $i$ is a rational homotopy equivalence and the least nonvanishing $p$-torsion homotopy group of $S^{2n+1}$ is $\pi_{2n+2p-2}(S^{2n+1})\cong \mathbb{Z}/p$. Therefore, $i$ is $(2n+2p-2)$-connected after localization at $p$.

For the second statement, note that the homotopy groups of $S^{2n+1}\{p^r\}$ have only $p$-torsions. 
By the long exact sequence of the homotopy groups of the homotopy fibration $S^{2n+1}\{p^r\} \stackrel{}{\larrow} S^{2n+1} \stackrel{p^r}{\larrow} S^{2n+1}$, the least two nonvanishing $p$-torsion homotopy groups of $S^{2n}\{p^r\}$ are
\[
\pi_{2n}(S^{2n+1}\{p^r\}) \cong \mathbb{Z}/p^r, \ \ \  \pi_{2n+2p-3}(S^{2n+1}\{p^r\}) \cong \mathbb{Z}/p.
\]
It is clear that $i_p$ induces an isomorphism on $\pi_{2n}$, and hence is $(2n+2p-3)$-connected.
\end{proof}

%----------------
\subsection{A suspension splitting of finite complexes}

In this subsection, we prove Proposition \ref{Xwedge-prop1} in three steps. In particular, it provides an alternative proof of \cite[Lemma 5.1]{HT24}.

\begin{lemma} 
   \label{oddlift-lemma} 
   Let $X$ be a path-connected finite $CW$-complex of dimension $d$. Suppose that $x\in H^{2n+1}(X;\mathbb{Z})$ is represented by a map $f: X\stackrel{}{\larrow} K(\mathbb{Z}, 2n+1)$. Then after localization away from any prime $p$ satisfying $2n+2p-2<d$, there is a lift
   \[
   \diagram
   & S^{2n+1} \dto^{i} \\
   X\urto^{\widetilde{f}} \rto^<<<<<{f}  & K(\mathbb{Z}, 2n+1)
   \enddiagram
   \]
  for some map $\widetilde{f}$. 
\end{lemma} 
\begin{proof}
The obstructions to the lifting problem are some cohomology classes in $\widetilde{H}^{i}(X; \pi_{i-1}(F))$ with $i\leq d$, where $F$ is the homotopy fibre of $i$. By Lemma \ref{SKlemma}, $F$ is $(2n+2q-3)$-connected after localization at any prime $q$. 
Localize away from any prime $p$ satisfying $2n+2p-2<d$. It follows that $F$ is $(d-1)$-connected. Hence, $\widetilde{H}^{i\leq d}(X; \pi_{i-1}(F))=0$ and the obstructions vanish. 
\end{proof}

\begin{corollary}\label{S-coro}
Let $X$ be a path-connected finite $CW$-complex of dimension $d$. Suppose that $x\in H^{j+1}(\Sigma X;\mathbb{Z})$ is represented by a map $f: \Sigma X\stackrel{}{\larrow} K(\mathbb{Z}, j+1)$. Then after localization away from any prime $p$ satisfying $j+2p-3<d$, there is a lift
   \[
   \diagram
   & S^{j+1} \dto^{i} \\
  \Sigma X\urto^{\widetilde{f}} \rto^<<<<<{f}  & K(\mathbb{Z}, j+1)
   \enddiagram
   \]
  for some map $\widetilde{f}$. 
\end{corollary} 
\begin{proof}
If $j=2n$ is even, the corollary follows by applying Lemma \ref{oddlift-lemma} to $\Sigma X\stackrel{f}{\larrow} K(\mathbb{Z}, j+1)$. 
If $j=2n+1$ is odd, localize away from any prime $p$ satisfying $j+2p-3<d$. By applying Lemma \ref{oddlift-lemma} to the adjoint map $X \stackrel{f'}{\larrow}   \Omega K(\mathbb{Z}, j+1)\simeq K(\mathbb{Z}, j)$ of $f$, we obtain a lift 
 \[
   \diagram
   & S^{j} \dto^{i} \\
   X\urto^{\widetilde{f}'} \rto^<<<<<{f'}  & \Omega K(\mathbb{Z}, j+1)
   \enddiagram
   \]
for some map $\widetilde{f}'$. Taking adjoint, we see that $\widetilde{f}:=\Sigma \widetilde{f}'$ is a lift of $f$. 
\end{proof}

If $X$ is a path-connected finite $CW$-complex of dimension $d$ and connectivity $s$, 
let 
\begin{equation}\label{PX-eq}
\calP(X)=\{ \text{prime}~q~|~s+2q-2<d, ~\text{or $H_{\ast}(X;\mathbb{Z})$ has $q$-torsion} \}.
\end{equation}
Note that the finiteness condition on $X$ implies that 
$\mathcal{P}(X)$ is a finite set of primes. Additionally, if $X$ is a sphere $\calP(X)$ is an empty set. 

\begin{proposition}\cite[Lemma 5.1]{HT24} 
   \label{Xwedge-prop1} 
   Let $X$ be a path-connected finite $CW$-complex of dimension $d$ and connectivity $s$. 
   Then after localization away from $\calP(X)$, $\Sigma X$ is homotopy equivalent to a wedge of spheres. 
\end{proposition} 
\begin{proof}
Localize away from $\calP(X)$. Then $H^{\ast}(X;\mathbb{Z})$ is torsion free. Choose an additive basis $\{x_k\}$ of $\widetilde{H}^\ast(\Sigma X;\mathbb{Z})$ with a representative $f_k: \Sigma X\larrow K(\mathbb{Z}, j_k+1)$ for each $x_k$. Note that $j_k\geq s+1$.
By Corollary \ref{S-coro}, each $f_k: \Sigma X\larrow K(\mathbb{Z}, j_k+1)$ can be lifted to a map $\widetilde{f}_k: \Sigma X\larrow S^{j_k+1}$. Consider the composite
\[
\Phi: \Sigma X \stackrel{\mu'}{\larrow} \bigvee_{k} \Sigma X  \stackrel{\mathop{\bigvee}\limits_{k} \widetilde{f}_k}{\larrow}  \bigvee_{k} S^{j_k+1},
\]
where $\mu'$ is the iterated comultiplication. Since each $\widetilde{f}_k$ detects the cohomology class $x_k$, the map $\Phi$ induces an isomorphism on cohomology, and then induces an isomorphism on homology. Therefore, the Whitehead theorem implies that $\Phi$ is a homotopy equivalence.  
\end{proof}

%----------------
\subsection{A refined suspension splitting of finite complexes}
In this subsection, we prove Proposition \ref{Xwedge-prop2} in four steps.

\begin{lemma} 
   \label{evenlift-lemma} 
   Let $X$ be a path-connected finite $CW$-complex of dimension $d$. Suppose that $y\in H^{2n}(X;\mathbb{Z}/p^r)$ is represented by a map $g: X\stackrel{}{\larrow} K(\mathbb{Z}/p^r, 2n)$. Then after localization away from any prime $p$ satisfying $2n+2p-3<d$, there is a lift
  \[
   \diagram
   & S^{2n+1}\{p^r\} \dto^{i_p} \\
   X\urto^{\overline{g}} \rto^<<<<<{g}  & K(\mathbb{Z}/p^r, 2n)
   \enddiagram
   \]
  for some map $\overline{g}$ with $i_p$ defined in (\ref{ip-eq}). 
\end{lemma} 
\begin{proof}
The obstructions to the lifting problem are some cohomology classes in $\widetilde{H}^{i}(X; \pi_{i-1}(G))$ with $i\leq d$, where $G$ is the homotopy fibre of $i_p$. By Lemma \ref{SKlemma}, $G$ is $(2n+2q-4)$-connected after localization at any prime $q$. 
Localize away from any prime $p$ satisfying $2n+2p-3<d$. It follows that $G$ is $(d-1)$-connected. Hence, $\widetilde{H}^{i\leq d}(X; \pi_{i-1}(G))=0$ and the obstructions vanish. 
\end{proof}

Recall the Moore space $P^{2n+2}(p^r)$ is defined by the homotopy cofibration
\[
S^{2n+1}\stackrel{p^r}{\larrow} S^{2n+1}\stackrel{j}{\larrow} P^{2n+2}(p^r),
\]
where $j$ is the inclusion of the bottom cell. 
There is a map 
\[
\iota_p: P^{2n+2}(p^r)\larrow K(\mathbb{Z}/p^r, 2n+1)
\]
representing a generator of $H^{2n+1}(P^{2n+2}(p^r);\mathbb{Z}/p^r)\cong \mathbb{Z}/p^r$.
\begin{lemma}\label{evenlift-lemma2}
  Let $X$ be a path-connected finite $CW$-complex of dimension $d$. Suppose that $y\in H^{2n+1}(\Sigma X;\mathbb{Z}/p^r)$ is represented by a map $g: \Sigma X\stackrel{}{\larrow} K(\mathbb{Z}/p^r, 2n+1)$. Then after localization away from any prime $p$ satisfying $2n+2p-3<d$, there is a lift
 \[
   \diagram
   & P^{2n+2}(p^r) \dto^{\iota_p} \\
   \Sigma X\urto^{\widetilde{g}} \rto^<<<<<{g}  & K(\mathbb{Z}/p^r, 2n+1)
   \enddiagram
   \]
  for some map $\widetilde{g}$. 
\end{lemma} 
\begin{proof}
There is a homotopy fibration diagram 
\[
  \diagram 
       S^{2n+1}\{p^{r}\}\rto\dto^{s} & S^{2n+1}\rto^-{p^{r}}\dto 
           & S^{2n+1}\dto^{j} \\ 
       \Omega P^{2n+2}(p^{r})\rto & \ast\rto & P^{2n+2}(p^{r}) 
  \enddiagram 
\]  
where $s$ is an induced map of fibres. 
Localize away from any prime $p$ satisfying $2n+2p-3<d$. Consider the homotopy commutative diagram 
  \[
   \diagram
   & S^{2n+1}\{p^r\} \dto^{i_p}   \rto^{s}   & \Omega P^{2n+2}(p^r)  \dlto^{\Omega \iota_p }  \\
   X\urto^{\overline{g}'} \rto^<<<<<{g'}  & \Omega K(\mathbb{Z}/p^r, 2n+1)
   \enddiagram
   \]
  where $g'$ is the adjoint of $g$, the left triangle is obtained by Lemma \ref{evenlift-lemma}, and the right triangle homotopy commutes as $s$ induces an isomorphism on $H^{2n}(-;\mathbb{Z}/p^r)$. Then the adjoint of $s\circ \overline{g}'$ is a lift of $g$.
\end{proof}

\begin{corollary}\label{Sp-coro}
Let $X$ be a path-connected finite $CW$-complex of dimension $d$. Suppose that $y\in H^{j+1}(\Sigma^2 X;\mathbb{Z}/p^r)$ is represented by a map $g: \Sigma^2 X\stackrel{}{\larrow} K(\mathbb{Z}/p^r, j+1)$. Then after localization away from any prime $p$ satisfying $j+2p-4<d$, there is a lift
   \[
   \diagram
   & P^{j+2}(p^r) \dto^{\iota_p} \\
  \Sigma^2 X\urto^{\widetilde{g}} \rto^<<<<<{g}  & K(\mathbb{Z}/p^r, j+1)
   \enddiagram
   \]
  for some map $\widetilde{g}$. 
\end{corollary} 
\begin{proof}
If $j=2n$ is even, the corollary follows by applying Lemma \ref{evenlift-lemma2} to $\Sigma^2 X\stackrel{g}{\larrow} K(\mathbb{Z}/p^r, j+1)$. 
If $j=2n+1$ is odd, localize away from any prime $p$ satisfying $j+2p-4<d$. By applying Lemma \ref{evenlift-lemma2} to the adjoint map $\Sigma X \stackrel{g'}{\larrow}   \Omega K(\mathbb{Z}/p^r, j+1)\simeq K(\mathbb{Z}/p^r, j)$ of $g$, we obtain a lift 
 \[
   \diagram
   & P^{j+1}(p^r) \dto^{\iota_p} \\
   \Sigma X\urto^{\widetilde{g}'} \rto^<<<<<{g'}  & \Omega K(\mathbb{Z}/p^r, j+1)
   \enddiagram
   \]
for some map $\widetilde{g}'$. Taking adjoint, we see that $\widetilde{g}:=\Sigma \widetilde{g}'$ is a lift of $g$. 
\end{proof}

If $X$ is a path-connected finite $CW$-complex of dimension $d$ and connectivity $s$, 
let 
\begin{equation}\label{QX-eq}
\calQ(X)=\{ \text{prime}~q~|~s+2q-2<d\}.
\end{equation}
Note that $\mathcal{Q}(X)\subseteq \calP(X)$ is a finite set of primes. Additionally, if $X$ is a Moore space, $\calQ(X)$ is an empty set. 
Recall a sphere is a special example of Moore space. 

\begin{proposition} 
   \label{Xwedge-prop2} 
   Let $X$ be a path-connected finite $CW$-complex of dimension $d$ and connectivity $s$. 
   Then after localization away from $\calQ(X)$, $\Sigma^2 X$ is homotopy equivalent to a wedge of Moore spaces. 
\end{proposition} 
\begin{proof}
Localize away from $\calQ(X)$. Let $\{a_k, b_\ell\}\in \widetilde{H}_\ast(\Sigma^2 X;\mathbb{Z})$ be an additive basis such that each $a_k$ is of infinite order and each $b_\ell$ is of order $p_\ell^{r_\ell}$.

For the torsion free homology class $a_k$, denote by $x_k\in H^{j_k+1}(\Sigma^2 X;\mathbb{Z})$ its dual class. It can be represented by a map $f_k: \Sigma^2 X\larrow K(\mathbb{Z}, j_k+1)$. Note that $j_k\geq s+2$. 
By Corollary \ref{S-coro}, each $f_k: \Sigma^2 X\larrow K(\mathbb{Z}, j_k+1)$ can be lifted to a map $\widetilde{f}_k: \Sigma^2 X\larrow S^{j_k+1}$. In particular, $\widetilde{f}_k$ detects the homology class $a_k$.

For the torsion homology class $b_\ell$, denote by $y_\ell \in H^{t_\ell+1}(\Sigma^2 X;\mathbb{Z}/p_\ell^{r_\ell})$ the dual class of its mod-$p_\ell^{r_\ell}$ reduction. It can be represented by a map $g_\ell: \Sigma^2 X\larrow K(\mathbb{Z}/p_\ell^{r_\ell}, t_\ell+1)$. Note that $t_\ell\geq s+2$. 
By Corollary \ref{Sp-coro}, each $g_\ell: \Sigma^2 X\larrow K(\mathbb{Z}/p_\ell^{r_\ell}, t_\ell+1)$ can be lifted to a map $\widetilde{g}_\ell: \Sigma^2 X\larrow P^{t_\ell+2}(p_\ell^{r_\ell})$. In particular, $\widetilde{g}_\ell$ detects the homology class $b_\ell$. 

Consider the composite
\[
\Psi: \Sigma^2 X \stackrel{\mu'}{\larrow} \Big(\bigvee_{k} \Sigma X \Big)\vee \Big(\bigvee_{\ell} \Sigma X \Big) \stackrel{\Big(\mathop{\bigvee}\limits_{k} \widetilde{f}_k \Big) \vee \Big(\mathop{\bigvee}\limits_{\ell} \widetilde{g}_\ell\Big)}{\lllarrow}  \Big(\bigvee_{k} S^{j_k+1}\Big) \vee \Big(\bigvee_{\ell} P^{t_\ell+2}(p_\ell^{r_\ell})\Big),
\]
where $\mu'$ is the iterated comultiplication. Since $\widetilde{f}_k$ and $\widetilde{g}_\ell$ detect the homology class $a_k$ and $b_\ell$ respectively, the map $\Psi$ induces an isomorphism on homology. Therefore, the Whitehead theorem implies that $\Psi$ is a homotopy equivalence.  
\end{proof}

%---------------------------------------------------------------------------------------
\section{General homotopy cofibres}
\label{sec: gen}
In this section, we prove Theorems \ref{gen-hyper-thm} and \ref{gen-hyper-thm2}. For the homotopy cofibrations in their statements, the homotopy cofibres can be infinite.

%---------------------------
\subsection{Beben-Theriault's decomposition}
\label{sec: BT}
The following theorem proved in~\cite{BT22} will be frequently used. 
\begin{theorem} 
   \label{GTcofib} 
   Let 
   \(\nameddright{\Sigma A}{h}{X}{\varphi}{Y}\) 
   be a homotopy cofibration. Suppose that $h$ is inert. Then there there is a homotopy fibration 
   \[\llnameddright{(\Sigma\Omega Y\wedge A)\vee\Sigma A}{} 
           {X}{\varphi}{Y}\] 
   which splits after looping to give a homotopy equivalence 
   \[\hspace{5cm}\Omega X\simeq\Omega Y\times\Omega((\Sigma\Omega Y\wedge A)\vee\Sigma A). 
         \hspace{5cm}\Box\] 
\end{theorem}

%---------------------------
\subsection{Proof of Theorem \ref{gen-hyper-thm}}
\label{sec: gen-hyper-thm}
Let us start with two lemmas. 
\begin{lemma}\label{Guylemma}
If there is a loop wedge $\Omega (S^a \vee S^b)$ retracting off $\Omega X$ with $a$, $b\geq 2$, then $X$ is $\mathbb{Z}/p^r$-hyperbolic for all primes $p$ and all $r\geq 1$.

If further $X$ is of finite type and has finite rational Lusternik-Schnirelmann category, then $X$ is rationally hyperbolic. 
\end{lemma}
\begin{proof}
By \cite[Theorem 1]{Boy22}, $S^{a}\vee S^{b}$ is $\mathbb{Z}/p^r$-hyperbolic for all primes $p$ and all $r\geq 1$. Then so is $X$. Also, it is well-known that $S^{a}\vee S^{b}$ is rationally hyperbolic, and so is $X$.
\end{proof}

\begin{lemma}\label{StYlemma}
Let $f: S^{t}\larrow Y$ be a map. If $\Omega f$ has a left homotopy inverse, then $S^t$ retracts off $\Sigma \Omega Y$.
\end{lemma}
\begin{proof}
Let $r: \Omega Y\larrow \Omega S^{t}$ be a left homotopy inverse $\Omega f$. Consider the composite
\[
S^{t}\stackrel{i}{\larrow}\Sigma \Omega S^{t}  \stackrel{\Sigma \Omega f}{\larrow} \Sigma \Omega Y\stackrel{\Sigma r}{\larrow}  \Sigma \Omega S^{t} \stackrel{{\rm ev}}{\larrow} S^t,
\] 
where $i$ is the bottom cell inclusion and ${\rm ev}$ is the evaluation map. It is homotopic to the identity map as ${\rm ev}\circ i$ is. Hence, the sphere $S^t$ retracts off $\Sigma \Omega Y$.
\end{proof}

Recall that the finite set $\calP(A)$ of primes is defined for any path connected finite $CW$-complex $A$ in (\ref{PX-eq}).

\begin{proof}[Proof of Theorem \ref{gen-hyper-thm}]
For the two cases of the theorem, we will prove there is a loop wedge $\Omega (S^a\vee S^b)$ retracting off $\Omega X$ in their respective local categories, and then the theorem will follow from Lemma \ref{Guylemma}.
By Theorem \ref{GTcofib}, there is a homotopy equivalence
\[
\Omega X\simeq\Omega Y\times\Omega((\Sigma\Omega Y\wedge A)\vee\Sigma A)
\]
after localization away from $P_1$. 

(1). Localize away from $\calP(A)\cup P_1$. Since $\Sigma A$ is finite, Proposition \ref{Xwedge-prop1} implies that $\Sigma A$ is homotopy equivalent to a wedge of spheres. By the assumption ${\rm dim}(\widetilde{H}^\ast(A;\mathbb{Q}))\geq 2$, there is a wedge of simply connected spheres $S^a \vee S^b$ retracting off $\Sigma A$. Then $\Omega (S^a\vee S^b)$ retracts off $\Omega X$ by the above loop space decomposition.

(2). Localize away from $\calP(A)\cup P_1\cup P_2$. When ${\rm dim}(\widetilde{H}^\ast(A;\mathbb{Q}))\geq 2$, the theorem is proved in (1). Hence, we can suppose that ${\rm dim}(\widetilde{H}^\ast(A;\mathbb{Q}))=1$. Then Proposition \ref{Xwedge-prop1} implies that $\Sigma A\simeq S^a$ for some $a\geq 2$. By assumption, there is a map $S^t\stackrel{}{\larrow} Y$ whose loop has a left homotopy inverse. It follows from Lemma \ref{StYlemma} that $S^t$ retracts off $\Sigma \Omega Y$. Hence there is a sphere $S^b$ retracting off $\Omega Y\wedge S^a$ with $b=a+t-1$. It follows that $S^a\vee S^b$ retracts off $(\Sigma\Omega Y\wedge A)\vee\Sigma A$, and then $\Omega (S^a\vee S^b)$ retracts off $\Omega X$ by the above loop space decomposition.
\end{proof}

%---------------------------
\subsection{Proof of Theorem \ref{gen-hyper-thm2}}
\label{sec: gen-hyper-thm2}

Recall that the finite sets $\calP(A)$, $\calQ(A)$ of primes are defined for any path connected finite $CW$-complex $A$ in (\ref{PX-eq}), (\ref{QX-eq}), respectively.

\begin{proof}[Proof of Theorem \ref{gen-hyper-thm2}]

Localize at the odd prime $p$. By Theorem \ref{GTcofib}, there is a homotopy equivalence
\[
\Omega X\simeq\Omega Y\times\Omega((\Sigma\Omega Y\wedge A)\vee\Sigma A).
\]

(a). Since $\Sigma A$ is finite and $p\notin\mathcal{P}(A)$, Proposition \ref{Xwedge-prop1} implies that $\Sigma A$ is homotopy equivalent to a wedge of spheres. As $\Sigma A$ is not rational contractible, there is a sphere $S^a$ retracting off $\Sigma A$, and then $\Sigma^a \Omega Y$ retracts off $\Sigma\Omega Y\wedge A$. Therefore, $\Omega \Sigma^a \Omega Y$ retracts off $\Omega X$ by the above loop space decomposition. To prove that $X$ is $\mathbb{Z}/p^r$-hyperbolic, it suffices to show that $\Sigma^a \Omega Y$ is $\mathbb{Z}/p^r$-hyperbolic.

Recall for any path connected co-$H$-space $W$, there is a map $W \stackrel{i}{\larrow} \Sigma\Omega W$ which is a right homotopy inverse of the evaluation map ${\rm ev}: \Sigma \Omega W\larrow W$.
Consider the homotopy commutative diagram
\[
\diagram
P^{k+1}(p^r) \rto^<<<{i} \drdouble  & \Sigma\Omega P^{k+1}(p^r) \rto^<<<{\Sigma \Omega \mu} \dto^{{\rm ev}} & 
\Sigma \Omega Y \dto^{{\rm ev}} \\
&  P^{k+1}(p^r)  \rto^<<<<<{\mu}  & Y,
\enddiagram
\]
where the right square commutes by the naturality of the evaluation map. Since $\mu$ induces an injection on homology with $\mathbb{Z}/p^r$-coefficient by assumption, so does the composite $\Sigma \Omega \mu\circ i$ on the top row. 
It follows that $\Sigma^{a-1}(\Sigma \Omega \mu\circ i): P^{a+k}(p^r)\larrow \Sigma^a\Omega Y$ induces an injection on homology with $\mathbb{Z}/p^r$-coefficient. Then \cite[Theorem 1.6]{Boy24} implies that $\Sigma^a \Omega Y$ is $\mathbb{Z}/p^r$-hyperbolic.

(b). By assumption, there is a map $S^t\stackrel{}{\larrow} Y$ whose loop has a left homotopy inverse. It follows from Lemma \ref{StYlemma} that $S^t$ retracts off $\Sigma \Omega Y$, and then $\Sigma^t A$ retracts off $\Sigma \Omega Y\wedge A$. Therefore, $\Omega \Sigma^t A$ retracts off $\Omega X$ by the above loop space decomposition of $\Omega X$. To prove $X$ is $\mathbb{Z}/p^r$-hyperbolic, it suffices to show that $\Sigma^t A$ is $\mathbb{Z}/p^r$-hyperbolic.

As $Y$ is simply connected, we see that $t\geq 2$. Since $\Sigma^t A$ is finite and $p\notin\mathcal{Q}(A)$, Proposition \ref{Xwedge-prop2} implies that $\Sigma^t A$ is homotopy equivalent to a wedge of Moore spaces. As there is an order $p^r$ element in $H_\ast(A;\mathbb{Z})$ by assumption, there is a Moore space $P^{j+1}(p^s)$ retracting off $\Sigma^t A$ for some $s\geq r$ and $j\geq t\geq 2$. By \cite[Theorem 1.3]{Boy24}, $P^{j+1}(p^s)$ is $\mathbb{Z}/p^r$-hyperbolic, and so is $\Sigma^t A$.
\end{proof}

%---------------------------------------------------------------------------------------
\section{Elliptic homotopy cofibres I: an elementary case}
\label{sec: elliptic}
In this section, we prove Theorems \ref{elliptic-hyper-thm-intro}. For the homotopy cofibration in its statement, the homotopy cofibres is assumed to be rationally elliptic.

%----------------------------------
\subsection{McGibbon-Wilkerson's theorem}

In 1986, McGibbon and Wilkerson~\cite{MW86} showed the following classical result, which implies that any finite rationally elliptic space, after looping, is a finite product of spheres and loops on spheres provided 
enough primes have been inverted. 

\begin{theorem} 
   \label{McW} 
   Let $X$ be a simply-connected finite rationally elliptic $CW$-complex. Then after localization away from finitely many primes, there is a homotopy equivalence 
   \[\Omega X\simeq\mathop{\prod}\limits_{i\in\mathcal{I}} S^{2m_{i}-1}\times\mathop{\prod}\limits_{j\in\mathcal{J}} \Omega S^{2n_{j}-1}\] 
   for finite index sets $\mathcal{I}$ and $\mathcal{J}$.  ~$\qqed$
\end{theorem} 

\begin{remark}\label{McWremark}
The homotopy equivalence in Theorem \ref{McW} can be chosen such that it preservers the loop structure on each $\Omega S^{2n_{j}-1}$-factor. Indeed, in the proof of Theorem \ref{McW} (for instance see \cite[Section 32.f]{FHT01} or \cite[Section 8.3.4]{HT25}), in order to retract off the $\Omega S^{2n_{j}-1}$-factor, a homotopy fibration
\[
F\larrow X\stackrel{b}{\larrow} S^{2n_{j}-1}
\] 
is constructed such that $b$ has a right homotopy inverse $S^{2n_{j}-1}\stackrel{\iota}{\larrow} X$ after localization away from finitely many primes. Then the looped composite $\Omega S^{2n_{j}-1} \stackrel{\Omega \iota}{\larrow}\Omega X\stackrel{\Omega b}{\larrow}\Omega S^{2n_{j}-1}$ implies a loop space decomposition $\Omega X\simeq \Omega  S^{2n_{j}-1}\times \Omega F$ that preservers the loop structure on the $\Omega S^{2n_{j}-1}$-factor.
\end{remark}

%----------------------------------
\subsection{Approximating rational inertness by local inertness}
The main results in this subsection, Proposition \ref{Qpinverse-prop} and its corollary, allow us to approximate rational inertness by local inertness at large primes under the assumption of rational ellipticity. This is the key step to prove Theorem \ref{elliptic-hyper-thm-intro}.

The proof of Proposition \ref{Qpinverse-prop} is divided into several steps. 
For a map $f: X\stackrel{}{\larrow} Y$ between simply connected $CW$-complexes, we denoted by $f_\mathbb{Q}: X_\mathbb{Q}\stackrel{}{\larrow} Y_\mathbb{Q}$ its rationalization. 

\begin{lemma}\label{Sinverselemma}
Let $X\stackrel{g}{\larrow} S^{2m-1}$ be a map of simply connected $CW$-complexes. If $g_\mathbb{Q}$ has a right homotopy inverse $S^{2m-1}_\mathbb{Q}\stackrel{f}{\larrow} X_\mathbb{Q}$, then there is a map $S^{2m-1}\stackrel{f'}{\larrow} X$ which is a right homotopy inverse of $g$ after localization away from finitely many primes.
\end{lemma}
\begin{proof}
Since the composite $S^{2m-1}_\mathbb{Q}\stackrel{f}{\larrow} X_\mathbb{Q}\stackrel{g_\mathbb{Q}}{\larrow} S^{2m-1}_{\mathbb{Q}}$ is homotopic to the identity map, the map $f$ generates a $\mathbb{Q}$-summand of $\pi_{2m-1}(X)\otimes \mathbb{Q}$. It follows that there is a map $f'': S^{2m-1}\stackrel{}{\larrow} X$ generating a $\mathbb{Z}$-summand of $\pi_{2m-1}(X)$ with $[f]=x[f''_\mathbb{Q}]$ for some $x=\frac{a}{b}\in \mathbb{Q}$. 
Consider the composite $S^{2m-1}\stackrel{f''}{\larrow} X\stackrel{g}{\larrow} S^{2m-1}$. Since $[g_\mathbb{Q}\circ f''_\mathbb{Q}]=\frac{b}{a}[g_\mathbb{Q}\circ f]=\frac{b}{a}$, we see that $g\circ f''$ is of degree $\frac{b}{a}$. Hence $a=\pm 1$ as the degree has to be an integer. It follows that $g\circ f''$ is a homotopy equivalence after localization away from all the primes that divides $b$. Precomposing with a homotopy inverse of $g\circ f''$, we see that $f'=f''\circ (g\circ f'')^{-1}$ is a right homotopy inverse of $g$ after localization away from the primes that divides $b$.
\end{proof}

\begin{lemma}\label{loopSinverselemma}
Let $X\stackrel{ g}{\larrow} S^{2n-1}$ be a map of simply connected $CW$-complexes. If $\Omega g_\mathbb{Q}$ has a right homotopy inverse $\Omega S^{2n-1}_\mathbb{Q}\stackrel{f}{\larrow} \Omega X_\mathbb{Q}$, then there is a map $\Omega S^{2n-1}\stackrel{f'}{\larrow} \Omega X$ which is a right homotopy inverse of $\Omega g$ after localization away from finitely many primes.
\end{lemma}
\begin{proof}
Let $E: S^{2n-2}\stackrel{}{\larrow} \Omega S^{2n-1}$ be the suspension map, that is, the inclusion of the bottom cell. It represents a generator of $\pi_{2n-2}(\Omega S^{2n-1})\cong \mathbb{Z}$. Further, by the universal property of loop suspension, the $H$-extension $\Omega S^{2n-1} \larrow \Omega S^{2n-1}$ of $E$ is homotopic to the identity map. 

Since $f$ is a right homotopy inverse of $\Omega g_\mathbb{Q}$, the composite $S^{2n-2}_\mathbb{Q}\stackrel{E_\mathbb{Q}}{\larrow} \Omega S^{2n-1}_\mathbb{Q} \stackrel{f}{\larrow} \Omega X_\mathbb{Q}\stackrel{\Omega g_\mathbb{Q}}{\larrow} \Omega S^{2n-1}_{\mathbb{Q}}$ is homotopic to $E_\mathbb{Q}$. 
Then the composite $f\circ E_\mathbb{Q}$ generates a $\mathbb{Q}$-summand of $\pi_{2n-2}(\Omega X)\otimes \mathbb{Q}$. It follows that there is a map $h'': S^{2n-2}\stackrel{}{\larrow} \Omega X$ generating a $\mathbb{Z}$-summand of $\pi_{2n-2}(\Omega X)$ with $[f\circ E_\mathbb{Q}]=x[h''_\mathbb{Q}]$ for some $x=\frac{a}{b}\in \mathbb{Q}$. 
Consider the composite $S^{2n-2}\stackrel{h''}{\larrow} \Omega X\stackrel{\Omega g}{\larrow} \Omega S^{2n-1}$. Since $[\Omega g_\mathbb{Q}\circ h''_\mathbb{Q}]=\frac{b}{a}[\Omega g_\mathbb{Q}\circ f\circ E_\mathbb{Q}]=\frac{b}{a}[E_\mathbb{Q}]$, we see that $[\Omega g\circ h'']=\pm b[E]\in \pi_{2n-2}(\Omega S^{2n-1})\cong \mathbb{Z}\{[E]\}$. 

Localizing away from the primes that divide $b$, there is a degree $\pm\frac{ 1}{b}$ self map $\nu_p$ of $S^{2n-1}$. Composing it with $\Omega g\circ h''$, we see that the composite $S^{2n-2}\stackrel{h''\circ \nu_p}{\larrow} \Omega X\stackrel{\Omega g}{\larrow} \Omega S^{2n-1}$ is homotopic to the suspension map $E$. Let $\Omega S^{2n-1}\stackrel{f'}{\larrow} \Omega X$ be the $H$-extension of $h''\circ \nu_p$. The naturality of the $H$-extension implies that the composite $\Omega S^{2n-1}\stackrel{f'}{\larrow} \Omega X\stackrel{\Omega g}{\larrow} \Omega S^{2n-1}$ is the $H$-extension of $\Omega g\circ h''\circ \nu_p\simeq E$, and hence is homotopic to the identity map. 
\end{proof}

Combing Lemmas \ref{Sinverselemma} and \ref{loopSinverselemma}, we can prove the following lemma. 

\begin{lemma}\label{S+loopSinverselemma}
Let $\Omega X\stackrel{\varphi}{\larrow} Z$ be a map of simply connected $CW$-complexes, where $Z=\mathop{\prod}\limits_{i\in\mathcal{I}} S^{2m_{i}-1}\times\mathop{\prod}\limits_{j\in\mathcal{J}} \Omega S^{2n_{j}-1}$
   for finite index sets $\mathcal{I}$ and $\mathcal{J}$. Suppose that $\varphi$ is a loop map on each $\Omega S^{2n_{j}-1}$-factor. 
If $\varphi_\mathbb{Q}$ has a right homotopy inverse, then $\varphi$ has a right homotopy inverse after localization away from finitely many primes. 
\end{lemma}
\begin{proof}
Let $S^{2m_{i}-1}\stackrel{\iota}{\larrow}  \mathop{\prod}\limits_{i\in\mathcal{I}} S^{2m_{i}-1}\times\mathop{\prod}\limits_{j\in\mathcal{J}} \Omega S^{2n_{j}-1}$ and $\mathop{\prod}\limits_{i\in\mathcal{I}} S^{2m_{i}-1}\times\mathop{\prod}\limits_{j\in\mathcal{J}} \Omega S^{2n_{j}-1}\stackrel{\pi}{\larrow} S^{2m_{i}-1}$ be the inclusion and projection of the $S^{2m_{i}-1}$-factor respectively. The composite of rational maps
\[
S^{2m_{i}-1}_\mathbb{Q}\stackrel{s\circ  \iota_\mathbb{Q}}{\lllarrow}\Omega X_\mathbb{Q}\stackrel{(\pi\circ \varphi)_\mathbb{Q}}{\lllarrow}S^{2m_{i}-1}_\mathbb{Q}
\]
is homotopic to the identity map, where $Z_\mathbb{Q}\stackrel{s}{\larrow}\Omega X_\mathbb{Q}$ is a right homotopy inverse of $\varphi_\mathbb{Q}$. By Lemma \ref{Sinverselemma}, there is a map $S^{2m_i-1}\stackrel{f_i}{\larrow} \Omega X$ which is a right homotopy inverse of $\pi\circ \varphi$ after localization away from a finite set $P_1$ of primes.

Similarly, let $\Omega S^{2n_{j}-1}\stackrel{\jmath}{\larrow}  \mathop{\prod}\limits_{i\in\mathcal{I}} S^{2m_{i}-1}\times\mathop{\prod}\limits_{j\in\mathcal{J}} \Omega S^{2n_{j}-1}$ and $\mathop{\prod}\limits_{i\in\mathcal{I}} S^{2m_{i}-1}\times\mathop{\prod}\limits_{j\in\mathcal{J}} \Omega S^{2n_{j}-1}\stackrel{\pi}{\larrow} \Omega S^{2n_{j}-1}$ be the inclusion and projection of the $\Omega S^{2n_{j}-1}$-factor respectively. The composite of rational maps
\[
\Omega S^{2n_{j}-1}_\mathbb{Q}\stackrel{s\circ \jmath_\mathbb{Q}}{\lllarrow}\Omega X_\mathbb{Q}\stackrel{(\pi \circ \varphi)_\mathbb{Q}}{\lllarrow} \Omega S^{2n_{j}-1}_\mathbb{Q}
\]
is homotopic to the identity map. By assumption $\pi \circ \varphi$ is a loop map. 
By Lemma \ref{loopSinverselemma}, there is a map $\Omega S^{2n_j-1}\stackrel{g_j}{\larrow} \Omega X$ which is a right homotopy inverse of $\pi\circ \varphi$ after localization away from a finite set $P_2$ of primes with $P_1\subseteq P_2$.

Localize away from $P_2$. 
Using the loop structure of $\Omega X$, we can multiply the maps $f_i$ and $g_j$ together to obtain a map 
\[
t: \mathop{\prod}\limits_{i\in\mathcal{I}} S^{2m_{i}-1}\times\mathop{\prod}\limits_{j\in\mathcal{J}} \Omega S^{2n_{j}-1} \stackrel{\mathop{\prod}\limits_{i\in\mathcal{I}} f_i\times\mathop{\prod}\limits_{j\in\mathcal{J}} g_j}{\lllarrow}  \mathop{\prod}\limits_{i\in\mathcal{I}} \Omega X \times\mathop{\prod}\limits_{j\in\mathcal{J}} \Omega X
\stackrel{}{\larrow}\Omega X.
\]
 Consider the composite 
\[
\mathop{\prod}\limits_{i\in\mathcal{I}} S^{2m_{i}-1}\times\mathop{\prod}\limits_{j\in\mathcal{J}} \Omega S^{2n_{j}-1}  \stackrel{t}{\larrow} \Omega X\stackrel{\varphi}{\larrow} \mathop{\prod}\limits_{i\in\mathcal{I}} S^{2m_{i}-1}\times\mathop{\prod}\limits_{j\in\mathcal{J}} \Omega S^{2n_{j}-1}.
\]
By the previous argument, it is homotopic to the identity map on each $S^{2m_{i}-1}$- and $\Omega S^{2n_{j}-1}$-factor, and then induces an isomorphism of homology on each factor. It follows that the composite $\varphi\circ t$ induces an isomorphism on homology, and then is a homotopy equivalence by the Whitehead theorem. Therefore, $\varphi$ has a right homotopy inverse.
\end{proof}

Using McGibbon-Wilkerson's theorem, we can prove Proposition \ref{Qpinverse-prop} from Lemma \ref{S+loopSinverselemma}.

\begin{proposition}\label{Qpinverse-prop}
Let $X\stackrel{\varphi}{\larrow} Y$ be a map of simply connected $CW$-complexes. Suppose that $Y$ is finite and rationally elliptic. 
If $\Omega \varphi_\mathbb{Q}$ has a right homotopy inverse, then $\Omega\varphi$ has a right homotopy inverse after localization away from finitely many primes. 
\end{proposition}
\begin{proof}
Let $s: \Omega Y_\mathbb{Q}\stackrel{}{\larrow}\Omega X_\mathbb{Q}$ be a right homotopy inverse of $\Omega  \varphi_\mathbb{Q}: \Omega X_\mathbb{Q} \stackrel{}{\larrow} \Omega Y_\mathbb{Q}$. 
For the finite rationally elliptic complex $Y$, Theorem \ref{McW} implies that after localization away from a finite set $P_1$ of primes, there is a homotopy equivalence 
   \[\Phi: \Omega Y\stackrel{\simeq}{\larrow}\mathop{\prod}\limits_{i\in\mathcal{I}} S^{2m_{i}-1}\times\mathop{\prod}\limits_{j\in\mathcal{J}} \Omega S^{2n_{j}-1}\] 
   for finite index sets $\mathcal{I}$ and $\mathcal{J}$. Denote by $\Phi^{-1}$ a homotopy inverse of $\Phi$. By Remark \ref{McWremark}, we can suppose that both $\Phi$ and $\Phi^{-1}$ are loop maps on each $\Omega S^{2n_{j}-1}$-factor. Then so is $\Phi\circ \Omega \varphi$. 
Since the composite of rational maps 
\[
\big(\mathop{\prod}\limits_{i\in\mathcal{I}} S^{2m_{i}-1}\times\mathop{\prod}\limits_{j\in\mathcal{J}} \Omega S^{2n_{j}-1}\big)_\mathbb{Q}  \stackrel{\Phi^{-1}_\mathbb{Q}}{\larrow}
\Omega Y_\mathbb{Q}\stackrel{s}{\larrow} \Omega X_\mathbb{Q}\stackrel{\Omega \varphi_\mathbb{Q}}{\larrow} \Omega Y_\mathbb{Q}
\stackrel{\Phi_\mathbb{Q}}{\larrow}(\mathop{\prod}\limits_{i\in\mathcal{I}} S^{2m_{i}-1}\times\mathop{\prod}\limits_{j\in\mathcal{J}} \Omega S^{2n_{j}-1})_\mathbb{Q}
\]
is homotopic to the identity map, Lemma \ref{S+loopSinverselemma} implies that, after localization away from a finite set $P_2$ of primes with $P_1\subseteq P_2$, the composite $\Phi\circ \Omega \varphi$ has a right homotopy inverse, and then so does $\Omega \varphi$. 
\end{proof}

Proposition \ref{Qpinverse-prop} immediately implies the following corollary.

\begin{corollary}\label{Qpinert-coro}
Let $A\stackrel{h}{\larrow}X\stackrel{\varphi}{\larrow} Y$ be a homotopy cofibration of simply connected $CW$-complexes. Suppose that $Y$ is finite and rationally elliptic. 
If $h$ is rationally inert, then $h$ is inert after localization away from finitely many primes.  ~$\qqed$ 
\end{corollary}

%----------------------------------
\subsection{An elementary case}
We can now prove Theorem \ref{elliptic-hyper-thm-intro}, which can be applied to the context of Poincar\'{e} Duality complexes in the sequel. 

\begin{proof}[Proof of Theorem \ref{elliptic-hyper-thm-intro}]
Since the finite complex $Y$ is rationally elliptic and $h$ is rationally inert, Corollary \ref{Qpinert-coro} implies that $h$ is inert after localization away from a finite set $P_1$ of primes. By Theorem \ref{GTcofib}, there is a homotopy equivalence
\[
\Omega X\simeq\Omega Y\times\Omega((\Sigma\Omega Y\wedge A)\vee\Sigma A)
\]
after localization away from $P_1$. 

For the finite rationally elliptic complex $Y$, Theorem \ref{McW} implies that after localization away from a finite set $P_2$ of primes, there is a homotopy equivalence 
   \[\Omega Y\simeq\mathop{\prod}\limits_{i\in\mathcal{I}} S^{2m_{i}-1}\times\mathop{\prod}\limits_{j\in\mathcal{J}} \Omega S^{2n_{j}-1}\] 
   for finite index sets $\mathcal{I}$ and $\mathcal{J}$. Since $Y$ is not rationally contractible, the index set $\mathcal{I}\coprod \mathcal{J}$ is non-empty. 
After localization away from $P_2$, the James suspension splitting theorem implies that $\Sigma\Omega Y$ is homotopy equivalent to a wedge of simply connected spheres. Choose any wedge summand $S^t$ of $\Sigma \Omega Y$. Then $S^t$ retracts off $\Sigma\Omega Y$ after localization away from $P_2$.

Localize away from $\calP(A)\cup P_1\cup P_2$. Since $\Sigma A$ is finite, Proposition \ref{Xwedge-prop1} implies that $\Sigma A$ is homotopy equivalent to a wedge of spheres. As $\Sigma A$ is not rational contractible, there is a simply connected sphere $S^a$ retracting off $\Sigma A$, and then $\Sigma^a \Omega Y$ retracts off $\Sigma\Omega Y\wedge A$. By the previous argument $S^t$ retracts off $\Sigma\Omega Y$. It follows that $S^b$ retracts off $\Sigma\Omega Y\wedge A$ with $b=a+t-1$. Therefore, $S^a \vee S^b$ retracts off $(\Sigma\Omega Y\wedge A)\vee \Sigma A$, and then $\Omega (S^a \vee S^b)$ retracts off $\Omega X$ by the above loop space decomposition of $\Omega X$. By Lemma \ref{Guylemma} the theorem follows. 
\end{proof}

%---------------------------------------------------------------------------------------
\section{Elliptic homotopy cofibres II: Poincar\'{e} Duality complexes}
\label{sec: ellipticII}
In this section, we apply Theorem \ref{elliptic-hyper-thm-intro} to the context of Poincar\'{e} Duality complexes and prove Theorems \ref{Melliptic-hyper-thm-intro}, \ref{punct-thm} and \ref{NSthm}. For the homotopy cofibrations in their statements, the homotopy cofibres is assumed to be rationally elliptic. 
From now on, every simply connected $CW$-complex is assumed to be equipped with its minimal cell structure.

%----------------------
\subsection{Rational Poincar\'{e} Duality complexes}

Let $M$ be a simply connected finite $CW$-complex. If $M$ is rationally elliptic, a classical result of Halperin \cite{Hal77} shows that $M$ is a {\it rational Poincar\'{e} Duality complex}, that is, its rational cohomology $H^\ast(M;\mathbb{Q})$ satisfies Poincar\'{e} duality. In this case, there is a {\it rational fundamental class} $[M]\in H_m(M;\mathbb{Q})\cong \mathbb{Q}$ in the top nontrivial rational homology of $M$, and the integer $m$ is referred to as the {\it formal dimension} of $M$.

Suppose further that $M$ has dimension $m$, which coincides with its formal dimension. There is a unique $m$-cell of $M$ representing the rational fundamental class $[M]$. We may call it the {\it formal top cell} of $M$. Then there is a homotopy cofibration
\begin{equation}\label{formal-cof-eq}
S^{m-1}\stackrel{h}{\larrow}\overline{M}\stackrel{}{\larrow} M,
\end{equation}
where $\overline{M}$ is $M$ with the formal top cell removed, and $h$ is the attaching map for the formal top cell of $M$.

The following proposition is noteworthy in its own right. 
\begin{proposition}\label{Minert-prop}
Let $M$ be a simply connected rationally elliptic finite $CW$-complex whose dimension equals its formal dimension. Suppose that the rational cohomology ring $H^\ast(M;\mathbb{Q})$ is not generated by a single element. Then the attaching map for the formal top cell of $M$ is inert after localization away from finitely many primes.
\end{proposition}
\begin{proof}
Consider the homotopy cofibration \eqref{formal-cof-eq}. 
By \cite{HL87}, the map $h$ is rationally inert as $H^\ast(M;\mathbb{Q})$ is not generated by a single element. Since $M$ is finite and rationally elliptic, Corollary \ref{Qpinert-coro} implies that $h$ is inert after localization away from finitely many primes.
\end{proof}

\begin{theorem}\label{Melliptic-hyper-thm}
Let $M$ be a simply connected rationally elliptic finite $CW$-complex whose dimension equals its formal dimension. Suppose that the rational cohomology ring $H^\ast(M;\mathbb{Q})$ is not generated by a single element. 
Then $\overline{M}$ is rationally hyperbolic and $\mathbb{Z}/p^r$-hyperbolic for almost all primes $p$ and all $r\geq 1$.

In particular, Moore's conjecture holds for $\overline{M}$ for all but finitely many primes $p$.
\end{theorem}
\begin{proof}
Consider the homotopy cofibration \eqref{formal-cof-eq}. Since $M$ is rationally elliptic and $H^\ast(M;\mathbb{Q})$ is not generated by a single element, $M$ is not rationally contractible, and Proposition \ref{Minert-prop} implies that $h$ is inert after localization away from finitely many primes. Then the theorem follows immediately from Theorem \ref{elliptic-hyper-thm-intro}. 
\end{proof}

%----------------------
\subsection{Poincar\'{e} Duality complexes}

Theorem \ref{Melliptic-hyper-thm-intro} is a special case of Theorem \ref{Melliptic-hyper-thm}.

\begin{proof}[Proof of Theorem \ref{Melliptic-hyper-thm-intro}]
Since $M$ is a Poincaré duality complex, its dimension coincides with its formal dimension; in particular, its unique top cell is the formal top cell. The homotopy cofibration \eqref{formal-cof-eq} thus corresponds to the attachment of the top cell to the lower skeleton $\overline{M}$. The theorem then follows directly from Theorem \ref{Melliptic-hyper-thm}.
\end{proof}

\begin{remark}\label{Mhyper-remark}
Let $M$ be a simply connected $m$-dimensional Poincar\'{e} Duality complex. Suppose that $M$ is not rationally contractible. By \cite[Corollary 5.3]{HT24}, if the attaching map for the top cell of $M$ is inert after localization away from finitely many primes, then $\overline{M}$ is rationally hyperbolic and $\mathbb{Z}/p^r$-hyperbolic for almost all primes $p$ and all $r\geq 1$. 

Therefore, if we can show that the attaching map for the top cell of a rationally hyperbolic Poincar\'{e} Duality complex is inert at large primes, Theorem \ref{Melliptic-hyper-thm-intro} will hold without the rationally elliptic condition. \end{remark}

%----------------------
\subsection{Punctured manifolds} A {\it punctured manifold} $M\backslash B$ is a manifold $M$ with a finite set $B$ of points removed.  

\begin{proof}[Proof of Theorem \ref{punct-thm}]
Suppose that $B$ is a one-point set. Then $M\backslash B\simeq \overline{M}$, and the theorem in this case follows immediately from Theorem \ref{Melliptic-hyper-thm-intro}. 

Otherwise, $B$ has more than one point. Then $M\backslash B \simeq \overline{M}\vee \mathop{\bigvee}\limits_{\#B-1} S^{m-1}$, where $m$ is the dimension of $M$. If $\#B=2$, the theorem in this case follows from the previous case. If $\#B\geq 3$, the theorem follows from Lemma \ref{Guylemma}. 
\end{proof}

%----------------------
\subsection{Manifolds with rationally spherical boundary}

The following lemma provides a family of examples of rational Poincar\'{e} Duality complexes, which are constructed from manifolds but are not necessarily manifolds.

\begin{lemma}\label{NMPDlemma}
Let $N$ be a simply connected manifold whose boundary is a rational sphere. Then the homotopy cofibre of the inclusion $i: \partial N\stackrel{}{\larrow} N$ of the boundary is a rational Poincar\'{e} Duality complex such that its dimension coincides with its formal dimension.
\end{lemma}
\begin{proof}
Suppose that $N$ is of dimension $m$. 
Let $M$ be the mapping cone of $i: \partial N\stackrel{}{\larrow} N$. It is an $m$-dimensional $CW$-complex. 
The inclusion map of pairs 
\[
j: (N,\partial N)\larrow (M, C(\partial N))
\] 
is a relative homeomorphism, where $C(\partial N)$ is the cone of $\partial N$. Then $j_\ast: H_{m}(N,\partial N)\larrow H_{m}(M, C(\partial N))$ is an isomorphism. Denote by $[M]=j_\ast([N])\in H_{m}(M;\mathbb{Z})\cong H_{m}(M, C(\partial N);\mathbb{Z})\cong  \mathbb{Z}$ the image of the fundamental class $[N]\in H_{m}(N, \partial N)$. The naturality of cap product implies a commutative diagram for (co)homology with arbitrary coefficient 
\[
\diagram
H^k(M, C(\partial N)) \dto^{-\cap [M]}\rto^<<<{j^\ast}_<<<{\cong}  & H^k(N, \partial N) \dto^{-\cap [N]}_{\cong}\\
H_{m-k}(M)    & H_{m-k}(N) \lto_{j_\ast}, 
\enddiagram
\]
where $-\cap [N]$ is an isomorphism by Poincar\'{e} duality. 

By assumption, $N$ is simply connected and $\partial N$ is connected. Then it is easy to see that $H^0(N, \partial N;\mathbb{Z})=0$ and $H^1(N, \partial N;\mathbb{Z})=0$, and hence $H_{m}(N;\mathbb{Z})=0$ and $H_{m-1}(N;\mathbb{Z})=0$ by Poincar\'{e} duality. Since $\partial N$ is a rational homotopy sphere, there is a rational homotopy cofibration
\[
S^{m-1}\stackrel{h}{\larrow} N\stackrel{j}{\larrow} M
\]
for some rational map $h$. From the long exact sequence of homology of a homotopy cofibration, it is straightforward to show that $j_\ast: H_{m-k}(N;\mathbb{Q})\larrow H_{m-k}(M;\mathbb{Q})$ is an isomorphism for any $k\geq 1$. Combing with the preceding cap product diagram, we see that 
\[
-\cap [M]: H^k(M;\mathbb{Q})\cong H^k(M, C (\partial N);\mathbb{Q}) \larrow H_{n-k}(M;\mathbb{Q})
\]
is an isomorphism for any $k\geq 1$. Additionally, it is clear that $-\cap [M]: H^0(M)\larrow H_{m}(M)$ is an isomorphism. Therefore, $M$ is a rational Poincar\'{e} duality complex with the rational fundamental class $[M]$, and its dimension coincides with its formal dimension.
\end{proof}

\begin{proof}[Proof of Theorem \ref{NSthm}]
Suppose that $N$ is of dimension $m$. 
Let $M$ be the homotopy cofibre of the inclusion $i: \partial N\stackrel{}{\larrow} N$ of the boundary. There is a homotopy cofibration
\[
\partial N\stackrel{i}{\larrow} N\stackrel{}{\larrow} M.
\]
 By assumption $\partial N\simeq_{\mathbb{Q}} S^{m-1}$. Then there is a map $f: S^{m-1}\larrow \partial N$ representing a generator of the rational homotopy group $\pi_{m-1}(\partial N)\otimes \mathbb{Q}\cong \mathbb{Q}$. By the Hurewicz theorem $H_\ast(f;\mathbb{Q})$ is an isomorphism. As $\partial N$ is finite and simply connected, this implies after localization away from a finite set $P$ of primes that $f$ induces an isomorphism on homology and hence is a homotopy equivalence by the Whitehead theorem.
 
 Localize away from the finite set $P$. The previous homotopy cofibration can be replaced by a  homotopy cofibration
 \[
 S^{m-1}\stackrel{h}{\larrow} N\stackrel{}{\larrow} M
 \]
 with $h=i\circ f$. Since $\partial N\simeq_{\mathbb{Q}} S^{m-1}$, by Lemma \ref{NMPDlemma} $M$ is a simply connected rational Poincar\'{e} duality complex whose dimension coincides with its formal dimension and the above homotopy cofibration corresponds to the attachment of its formal top cell. 
  In particular, we have an isomorphism 
 \[
 \widetilde{H}^\ast(M;\mathbb{Q})\cong \widetilde{H}^\ast(N;\mathbb{Q})\oplus \mathbb{Q}\{[M]\},
 \] 
 and thus,  $H^\ast(M;\mathbb{Q})$ is not generated by a single element, just as $H^\ast(N;\mathbb{Q})$ is not by assumption. Combining with the assumption that $M$ is rationally elliptic, Proposition \ref{Minert-prop} implies that $h$ is inert after a further localization away from finitely many primes. Then the theorem follows directly from Theorem \ref{elliptic-hyper-thm-intro}. 
\end{proof}

%%% The bibliography %%%
\bibliographystyle{amsalpha}

\end{document}